\documentclass{amsart}
\usepackage{amsmath,amssymb,amscd,amsthm,verbatim,alltt,amsfonts,array}
\usepackage[english]{babel}
\usepackage{latexsym}
\usepackage{amssymb}
\usepackage{euscript}
\usepackage{graphicx}

\newtheorem{theorem}{Theorem}[section]
\theoremstyle{plain}

\newtheorem{conjecture}[theorem]{Conjecture}
\newtheorem{corollary}[theorem]{Corollary}

\newtheorem{lemma}[theorem]{Lemma}

\newtheorem{proposition}[theorem]{Proposition}

\numberwithin{equation}{section}

\begin{document}

\title[Supraposinormality and hyponormality for the . . . .]{Supraposinormality and hyponormality for the \\ generalized Ces\`{a}ro matrices of order two}

\author {H. C. Rhaly Jr.}
\address {1081 Buckley Drive\\Jackson, Mississippi   39206\\ U.S.A.}
\email {rhaly@member.ams.org}

\dedicatory{In memory of Evans Burnham Harrington (1925-1997)}


\begin{abstract}
It is well known that the generalized Ces\`{a}ro matrices of order one are hyponormal operators on $\ell^2$, and it has recently been shown that the Ces\`{a}ro matrix of order two is also hyponormal.  Here the relatively new concept of supraposinormality is used to show that the generalized Ces\`{a}ro matrices of order two are both posinormal and coposinormal, and that ``most'' of them are also hyponormal.  A conjecture is propounded that would extend the hyponormality result.
\end{abstract}

\maketitle

\renewcommand{\labelenumi}{(\alph{enumi})}

\section{Introduction}

The concept of supraposinormality for operators on a Hilbert space $H$ was recently introduced and investigated in [\textbf{8}].  The operator $A  \in \mathcal{B}(H)$, the set of bounded linear operators on $H$, is {\it supraposinormal} if there exist positive operators $P$ and $Q \in \mathcal{B}(H)$ such that  \[AQA^*= A^*PA,\] where at least one of $P$, $Q$ has dense range.  The ordered pair $(Q,P)$ is referred to as an \textit{interrupter pair} associated with $A$.  A \textit{posinormal} operator is a particular case of a supraposinormal with $Q=I$, and a \textit{coposinormal} operator is a particular case of a supraposinormal with $P=I$.  The following theorem contains some key facts about supraposinormal operators.

\begin{theorem}  Suppose $A \in \mathcal{B}(H)$ satisfies $AQA^*= A^*PA$ for positive operators $P, Q \in B(H)$. 

\begin{enumerate}

\item If $Q$ has dense range, then $A$ is supraposinormal and $Ker (A) \subseteq Ker (A^*)$, where $Ker (A) = \{f \in H: Af = 0 \}$.    

\item If $P$ has dense range, then $A$ is supraposinormal and $Ker (A^*) \subseteq Ker (A)$.    

\item If $Q$ is invertible, then the supraposinormal operator $A$ is posinormal.  

\item If $P$ is invertible, then the supraposinormal operator $A$ is coposinormal.

\item If $P$ and $Q$ are both invertible, then $A$ is both posinormal and coposinormal with $Ker (A) = Ker (A^*)$ and $Ran (A) = Ran (A^*)$, where $Ran (A) = \{g \in H : g = Af$ for some $f \in H\}$.

\end{enumerate}
\end{theorem}

\begin{proof}  See [\textbf{8}].
\end{proof}

\pagebreak

Recall that the operator $A \in \mathcal{B}(H)$ is \textit{hyponormal}  if \[< (A^*A - AA^*) f , f > \hspace{2mm}  \geq \hspace{2mm} 0\]  for all $f \in H$.   

\begin{theorem}   If $A$ is supraposinormal operator on $H$ with $AQA^*= A^*PA $ and  

\[Q \geq I \geq P \geq 0,\]   

\noindent then $A$ is hyponormal.
\end{theorem}

\begin{proof}   Assume the hypothesis is satisfied.  Then   \begin{align*}   \langle (A^*A - AA^*) f , f \rangle = \langle (A^*A - A^* P A + A Q A^* - AA^*) f , f \rangle   \end{align*} 
\begin{align*}   = \langle (I-P)Af,Af \rangle + \langle(Q-I)A^*f,A^*f \rangle  \geq 0   \end{align*}  

 \noindent for all f in $H$, so $A$ is hyponormal.
\end{proof}

In a pair of 1921 papers [\textbf{2},\textbf{3}] Hausdorff studied the class of lower triangular matrices whose entries are \begin{equation*} h_{ij} = \binom{i}{j} \Delta^{i-j} \mu_{j}, \textrm{   }  0 \leq j \leq i, \end{equation*}  where $ \{ \mu_n \}$ is a real or complex sequence, and $\Delta$ is the forward difference operator defined by \begin{equation*} \Delta \mu_k = \mu_k - \mu_{k+1}, \textrm{   } \Delta^{n+1}  \mu_k = \Delta(\Delta^n \mu_k ).  \end{equation*}
Almost forty years later, Endl [\textbf{1}] and Jakimovski [\textbf{4}] independently defined a generalization of the Hausdorff matrices as follows.  For each $\alpha \geq 0$, $H^{(\alpha)}$ is a lower triangular matrix with entries \begin{equation} h_{ij}^{(\alpha)} = \binom{i+ \alpha }{i-j} \Delta^{i-j} \mu_{j}, \textrm{   }  0 \leq j \leq i. \end{equation}  The case $\alpha = 0$ results in the original Hausdorff matrices.

If we use \begin{equation*} \mu_j = \int_0^1 t^{j+ \alpha} d \chi (t) \textrm{ \hspace{3mm}   where  \hspace{2mm}  } \chi (t) = 1 - (1-t)^{\beta} \end{equation*} in $(1.1)$ with $\alpha, \beta \geq 0$, we obtain the generalized Ces\`{a}ro matrices of order $\beta$, $(C^{(\alpha)}, \beta)$, given by
\begin{align} (C^{(\alpha)}, \beta)_{ij} = \binom{i+\alpha}{i-j} \int_0^1 t^{j+\alpha}(1-t)^{i-j} \beta (1-t)^{\beta - 1} dt   \end{align}   \begin{align*} = \frac{\Gamma(i+\alpha+1)}{\Gamma(i-j + 1) \Gamma(j + \alpha + 1)} \cdot  \frac{\beta \Gamma(j+\alpha+1) \Gamma(i-j + \beta)} {\Gamma(i + \alpha + \beta + 1).} \end{align*}  

For fixed $\alpha \geq 0$ and $\beta = 1$, one obtains the generalized Ces\`{a}ro matrices of order one, $(C^{(\alpha)},1)$, with entries \begin{equation*} (C^{(\alpha)},1)_{ij} = \left \{ \begin{array}{lll}
\frac{1}{i + 1 + \alpha}& for &  0 \leq j \leq i \\
0 & for & j > i .\end{array}\right. \end{equation*} for all $i$.   We note that $(C^{(\alpha)},1) \in \mathcal{B} (\ell^2)$ for all $\alpha > -1$, not just for $\alpha \geq 0$.  The next result first appeared in [\textbf{7}, Theorem $2.4$].

\begin{proposition}  $(C^{(\alpha)},1)$ is posinormal and coposinormal for all $\alpha > -1$, and $(C^{(\alpha)},1)$ is hyponormal for all $\alpha \geq 0$. 
\end{proposition} 

\begin{proof}    If $Q:\equiv diag \{1 + \alpha, 1, 1, 1, ....\}$ and $P:\equiv diag\{\frac{n+1+ \alpha}{n+2+ \alpha}: n = 0, 1, 2, 3, ....\}$,  it can be verified that \begin{equation*} (C^{(\alpha)},1) Q (C^{(\alpha)},1)^* =   (C^{(\alpha)},1)^* P(C^{(\alpha)},1) \end{equation*}  with both $P$ and $Q$ invertible for all $\alpha > -1$, so it follows from Theorem $1.1$ that $(C^{(\alpha)},1)$ is both posinormal and coposinormal for those values of $\alpha$.  Since $Q \geq I \geq P$ for all $\alpha \geq 0$, it follows from Theorem $1.2$ that $(C^{(\alpha)},1)$ is hyponormal for those $\alpha$.
\end{proof}

\noindent It should be noted that earlier proofs of hyponormality for this operator (see [\textbf{6}, \textbf{10}]) were much more computational, so it seems natural and desirable to look for opportunities to apply the supraposinormality approach to other operators.  With that in mind, we turn to display $(1.2)$ when $\beta = 2$ and obtain \begin{equation*} (C^{(\alpha)},2)_{ij}  = \left \{ \begin{array}{lll}
\frac{2(i + 1 -j )}{(i+ 1 + \alpha)(i+ 2 + \alpha)}& for &  0 \leq j \leq i \\
0 & for & j > i \end{array}\right. , \end{equation*}   the entries of the generalized Ces\`{a}ro matrices of order two.  Note again that these operators are in $\mathcal{B} (\ell^2)$ for all $\alpha > -1$ (not just for $\alpha \geq 0$).  Comparison of the entries verifies that the Ces\`{a}ro matrix $(C^{(0)},2)$ of order two (take $\alpha = 0$) is precisely the matrix studied in [\textbf{9}], where it was shown that, as a result of the relationship  \begin{equation*} (C^{(0)},2)(C^{(0)},2)^*=(C^{(0)},2)^* diag \Big \{\frac{(n+1)(n+2)}{(n+3)(n+4)}: n = 0, 1, 2, 3, .... \Big \} (C^{(0)},2), \end{equation*}  $(C^{(0)},2)$ is a hyponormal operator on $\ell^2$.  In view of the relationship displayed above, together with inspection of the result obtained for $(C^{(\alpha)},1)$ in Proposition $1.3$, we decide to make an adjustment in the interrupter to allow for $\alpha$ and call the result $P$, compute $(C^{(\alpha)},2)^*P(C^{(\alpha)},2)$, and then try to find a diagonal or ``nearly'' diagonal operator $Q$ such that \begin{equation*} (C^{(\alpha)},2)Q(C^{(\alpha)},2)^*= (C^{(\alpha)},2)^*P(C^{(\alpha)},2) \end{equation*}  in the hope that Theorem $1.2$ will apply.  That approach has led to the main results presented in the next section.

Some of the calculations in the following sections have been aided by the Sagemath software system [\textbf{11}].

\section{Main Results}

Suppose  $\alpha > -1$.  Under consideration here will be the matrix $M : \equiv (C^{(\alpha)},2)$ whose entries are given by 
\[m_{ij} = \left \{ \begin{array}{lll}
\frac{2 (i + 1 - j)}{(i+ 1 + \alpha)(i+ 2 + \alpha)}& for &  0 \leq j \leq i \\
0 & for & j > i .\end{array}\right.\]   
\noindent It is worth noting that the range of M contains all the $e_n$'s from the standard orthonormal basis for $\ell^2$ since \begin{equation*} M  \Big [ \frac{1}{2} (n+1+\alpha)(n+2+\alpha)(e_n-2e_{n+1}+e_{n+2}) \Big ] = e_n. \end{equation*}

\begin{theorem}  The operator $M \in \mathcal{B}(\ell^2)$ is supraposinormal for $\alpha > -1$.  \end{theorem}

\begin{proof} In view of the considerations mentioned in the introduction, we first take  \begin{equation*} P :\equiv diag \Big \{ \frac{(n + 1 + \alpha)(n + 2 + \alpha)}{(n + 3 + \alpha)(n + 4 + \alpha)} : n \geq 0 \Big \}, \end{equation*} and then compute  $M^*PM$.  For $j \geq i$, the ($i$, $j$)-entry of $M^*PM$ is \begin{equation*} \sum_{k=0}^{\infty} \frac{4(j + 1 - i + k)(k + 1)}{(j + 1+ k + \alpha)(j + 2 + k + \alpha)(j + 3 + k + \alpha)(j + 4 + k + \alpha)}. \end{equation*}

\noindent The series is telescoping, as can be seen by rewriting the summand as \begin{equation*} s(k) - s(k+1) \end{equation*} where  \begin{equation*} s(k) :\equiv   \frac{ak^2+bk+c}{(j + 1+ k +  \alpha)(j + 2+ k + \alpha)(j + 3+ k + \alpha)}  \end{equation*}
with \begin{equation*}  a=4, b = 2(3j + 5 -i + 2 \alpha), \textrm{ and } c = \frac{2}{3} (3j + 3 - i + 2 \alpha) (j + 3 + \alpha). \end{equation*}
Consequently, for $j \geq i$, the ($i$, $j$)-entry of $M^*PM$ in simplified form is \begin{equation*} s(0) = \frac{2(3j + 3 - i + 2 \alpha)}{ 3  (j + 1+ \alpha)(j + 2 + \alpha)} .\end{equation*}  The computations for $i > j$ are similar.

Suppose that $Q = [q_{ij}]$ is given by 

\begin{equation*} Q = \left(\begin{array}{ccccccc} \frac{(3+2\alpha)(1+\alpha)(2+\alpha)}{6} &- \frac{\alpha(1 + \alpha) (2 + \alpha)}{3} &0&0&0&\ldots\\ - \frac{\alpha(1 + \alpha) (2 + \alpha)}{3} & \frac{(3- \alpha+2 \alpha^2) (2 + \alpha)}{6} &0&0&0&\ldots\\0&0&1&0&0 & \ldots\\0&0&0&1& 0 &\ldots\\0&0&0&0&1 &\ldots\\\vdots&\vdots&\vdots&\vdots&\vdots&\ddots \end{array}\right);  \hspace{2mm} \textrm{that is,} \end{equation*}   $q_{00} =  \frac{(3+2\alpha)(1+\alpha)(2+\alpha)}{6}$, $q_{10}=q_{01} = - \frac{\alpha(1 + \alpha) (2 + \alpha)}{3}$,  $q_{11} =\frac{(3- \alpha+2 \alpha^2) (2 + \alpha)}{6}$,    $q_{ii} =1$ for all $i=j \geq 2$, and $q_{ij} = 0$ otherwise. If $X :\equiv QM^*$, then $X = [x_{ij}]$ has entries

\begin{equation*} x_{ij} = \left \{ \begin{array}{lll}
\frac{[3(j+1)+2 \alpha](1 + \alpha)(2 + \alpha )}{3(j+1 + \alpha )(j +2 + \alpha)} & for &  i=0 \\
\frac{[3j (1- \alpha) - 2 \alpha (1 + \alpha)] (2 + \alpha)}{3( j+1 + \alpha)(j +2 + \alpha)} & for &  i=1 \\
\frac{2 (j+1-i)}{(j+1+\alpha)(j+2+\alpha)} & for & j \geq i \geq 2 \\
0 & for & i \geq 2, j < i.  \end{array}\right. \end{equation*}    \vspace{1mm}

\noindent For $j \geq i \geq 2$, the ($i$,$j$)-entry of $MQM^*=MX$ is \begin{align*} \frac{2(i+1)}{(i+1+\alpha)(i+2+\alpha)} \cdot \frac{(3j+3+2\alpha)(1+\alpha)(2+\alpha)}{3(j+1+\alpha)(j+2+\alpha)} \end{align*}   \begin{align*}  +  \frac{2i}{(i+1+\alpha)(i+2+\alpha)} \cdot \frac{(3j-3j\alpha-2\alpha-2\alpha^2)(2+\alpha)}{3(j+1+\alpha)(j+2+\alpha)}    \end{align*}    \begin{align*}  + \sum_{k=1}^{i-1} \frac{2(i-k)}{(i+1+\alpha)(i+2+\alpha)} \cdot \frac{2(j-k)}{(j+1+\alpha)(j+2+\alpha)}   \end{align*}  \begin{align*} = \frac{2(3j + 3 - i + 2 \alpha)}{ 3  (j + 1+ \alpha)(j + 2 + \alpha)} .\end{align*} The cases $j \geq i = 0$ and $j \geq i = 1$ are left to the reader. The computations for $i \geq j$ are similar, so it follows that $M^*PM = MQM^*$.  It is clear that $P$ has dense range, so $M$ is supraposinormal.  \end{proof}

\begin{corollary}  The operator $M$ is both posinormal and coposinormal for \mbox{$\alpha > -1$.} \end{corollary}

\begin{proof} Since \begin{equation*} \det \left(\begin{array}{cc} \frac{(3+2\alpha)(1+\alpha)(2+\alpha)}{6} &- \frac{\alpha(1 + \alpha) (2 + \alpha)}{3} \\ - \frac{\alpha(1 + \alpha) (2 + \alpha)}{3} & \frac{(3- \alpha+2 \alpha^2) (2 + \alpha)}{6} \end{array}\right) = \frac{(1+\alpha)(2+\alpha)^2 (3+\alpha)}{12}, \end{equation*} it is clear that $Q$ is invertible for $\alpha > -1$.  Since $P$ is also clearly invertible for $\alpha > -1$, it follows from Theorem $1.1$ that $M$ is both posinormal and coposinormal when $\alpha > -1$.  \end{proof}

\begin{corollary}  If $n$ is a positive integer and  $\alpha > -1$, then $M^n$ is both posinormal and coposinormal.
\end{corollary}

\begin{proof}  This follows from Corollary $2.2$ and [\textbf{5}, Corollary $2.2$ (b)].
\end{proof}

\begin{corollary}  The operator $M \in B(\ell^2)$ is hyponormal for $\alpha \in \{0\} \cup [1, \infty)$.
\end{corollary}

\begin{proof}  It is clear that $I \geq P$ for all $\alpha \geq 0$.  If $Z_0$ is the first finite section of $Q-I$ and $Z_1$ is the second, then \begin{equation*} \det (Z_0) = \frac{\alpha (13 + 9 \alpha + 2 \alpha^2)}{6} \end{equation*} and  \begin{equation*} \det (Z_1) = \frac{ \alpha^2 (\alpha -1)(1+\alpha)}{12}, \end{equation*}  so $Q \geq I$ for $\alpha \geq 1$ or $\alpha = 0$.  It follows from Theorem $1.2$ that  $M$ is hyponormal for $\alpha \in \{0\} \cup [1, \infty)$.  \end{proof}

While it seems plausible that $M$ is also hyponormal for $0 < \alpha < 1$, that question remains open for now.  The next section propounds a conjecture regarding that case.

\section{Conjecture concerning the exceptional case: $0 < \alpha <1$}

This lemma is included to provide support for the conjecture to follow.   Whereas the computations of the previous section centered on the supraposinormality of $M$, the computations in this section center on its posinormality.  

\begin{lemma}  If $M$ is the generalized Ces\`{a}ro matrix of order $2$ for fixed $\alpha > -1$ and $B :\equiv [b_{ij}]_{i,j \geq 0}$ is the matrix defined by \[b_{ij} = \left \{ \begin{array}{lll}
\frac{2(j+1-3i - 2 \alpha)}{( j+3+ \alpha )( j+4+\alpha)} & if  &  j > i-2 \\
\frac{( j+1+ \alpha )( j+2+\alpha)}{( j+3+ \alpha )( j+4+\alpha)} & if  &  j = i-2 \\
0 & if & j < i-2 \end{array}\right. ,\] then 
\begin{enumerate} 
\item $B$ is a bounded linear operator on $\ell^2$, 
\item $BM = M^*$, and
\item the interrupter $P = B^*B = [p_{ij}]$ has entries given by  \begin{equation*}  \label{eq:display} p_{ij} = \left \{ \begin{array}{lll}
\frac{i^4+2(5+2 \alpha) i^3+(35+ 26 \alpha +6 \alpha^2)i^2}{( i +3 + \alpha)^2 ( i + 4 + \alpha)^2}   +    &       &     \\

\hspace{5mm}  \frac{(50+50 \alpha +34 \alpha^2 + 4 \alpha^3)i +\alpha^4+6 \alpha^3 +45 \alpha^2 +28 \alpha +24}{( i +3 + \alpha)^2 ( i + 4 + \alpha)^2}       &   if      & i=j;\\

-\frac{2 \alpha (11 + (5-\alpha)(i+j) + 2 i j - 5 \alpha + 2 \alpha^2)}{( i+3+ \alpha )( i+4+\alpha)( j+3+ \alpha )( j+4+\alpha)}     &   if    &    i \ne j. \end{array}\right. \end{equation*}
 
 \end{enumerate}
\end{lemma} 

\begin{proof}  \begin{enumerate} \item First note that for $i \leq j$, $b_{ij} \in [-4/( j + 4 + \alpha), 2/( j + 4 + \alpha)]$ for all $i$, so the upper triangular part $T$ of matrix $B$ is a bounded operator on $\ell^2$.  If $U$ denotes the unilateral shift, $W_1$ is the weighted shift with weight sequence   \begin{equation*} \Big \{ \frac{4(n+1+\alpha)}{(n+3+\alpha)(n+4+\alpha)} : n \geq 0  \Big \}  \end{equation*} and $W_2$ is the weighted shift with weight sequence \begin{equation*} \Big \{ \frac{(n+1+\alpha)(n+2+\alpha)}{(n+3+\alpha)(n+4+\alpha)} : n \geq 0 \Big \}, \end{equation*} then \[B = T - W_1 +U W_2,\] so $B \in  \mathcal{B}(\ell^2)$.

\item   For $j \geq i$, the $(i,j)$-entry in the product matrix $BM$ is given by  \begin{equation*}  \sum_{k=0}^{\infty} \frac{4(k+1)(j+1-3i+k- 2 \alpha)}{( j + 1 + k + \alpha)( j + 2 + k + \alpha)( j + 3 + k + \alpha)( j + 4 + k + \alpha)}. \end{equation*}   The summand in the series can be expressed as \begin{equation*} s(k) - s(k+1) \end{equation*}  where  \begin{equation*} s(k) = \frac{4k^2+(10+ 6j-6i)k+2(j+1-i)(j+3+\alpha)}{( j + 1 + k + \alpha)( j + 2 + k + \alpha)( j + 3 + k + \alpha)}. \end{equation*}  The series is telescoping and converges to \begin{equation*} s(0) = \frac{2(j+1-i)}{( j + 1 + \alpha)(j + 2 + \alpha)}, \end{equation*}  which is exactly the $(i,j)$-entry of $M^*$ for $j \geq i$.

\noindent

\noindent For $j < i$, the $(i,j)$-entry of $BM$ is given by \begin{equation*} \frac{(i -1 + \alpha)(i + \alpha)}{(i +1 + \alpha)( i + 2 + \alpha)} \cdot \frac{2(i-j-1)}{(i -1 + \alpha)(i + \alpha)} \end{equation*} \begin{equation} + \frac{-4(i + \alpha)}{(i + 2 + \alpha)(i + 3 + \alpha)}  \cdot \frac{2(i-j)}{(i+ \alpha )( i +1+ \alpha)} \end{equation}  \begin{equation*}  + \sum_{k=0}^{\infty} \frac{4(i+1-j+k)(k + 1 -2i-2 \alpha)}{( i +1 + k + \alpha)( i +2 + k + \alpha)( i +3 + k + \alpha)( i +4 + k + \alpha)}.  \end{equation*}  The summand in the infinite series can be expressed as \begin{equation*} s(k) - s(k+1) \end{equation*}  where  \begin{equation*} s(k) = \frac{4k^2+(10+2i-2j)k +c}{( i +1 + k + \alpha)( i +2 + k + \alpha)( i +3 + k + \alpha)} , \end{equation*} with \begin{equation*} c = 6+(4-2 \alpha )i-(2-2 \alpha)j +2ij-2i^2+2 \alpha , \end{equation*} so the series is telescoping and converges to \begin{equation*} s(0) = \frac{6+(4-2\alpha)i-(2-2\alpha)j+2ij-2i^2+2\alpha}{(i + 1 + \alpha)(i + 2 + \alpha)(i + 3 + \alpha)}. \end{equation*} It is then straightforward to return to $(3.1)$ and show that  the $(i,j)$-entry of $BM$  for $j < i$ is \begin{equation*} \frac{2(i-j-1)}{(i + 1 + \alpha)(i + 2 + \alpha)} - \frac{8(i-j)}{(i + 1 + \alpha)(i + 2 + \alpha)(i + 3 + \alpha))}  + s(0) = 0, \end{equation*} which is precisely the $(i,j)$-entry of $M^*$ for $j < i$.

\item  For $i=j$, the $(i,j)$-entry in the product matrix $B^*B$ is given by  \begin{equation*}  \sum_{k=0}^{i+1} \frac{4(i+1-3k - 2 \alpha)^2}{(i + 3+\alpha)^2(i + 4+\alpha)^2}  + \frac{(i + 1+\alpha)^2(i + 2+\alpha)^2}{(i + 3+\alpha)^2(i + 4+\alpha)^2} = \end{equation*}   \begin{align*}  \frac{i^4+2(5+2 \alpha) i^3+(35+ 26 \alpha +6 \alpha^2)i^2 + (50+50 \alpha +34 \alpha^2 + 4 \alpha^3)i}{(i + 3+\alpha)^2(i + 4+\alpha)^2}  \end{align*} \begin{align*} +\frac{\alpha^4+6 \alpha^3 +45 \alpha^2 +28 \alpha +24}{(i + 3+\alpha)^2(i + 4+\alpha)^2} .\end{align*}

\noindent For $j \geq i+1$, the $(i,j)$-entry of $B^*B$ is given by   \begin{align*}    \sum_{k=0}^{i+1} \frac{2(i+1-3k-2 \alpha)}{(i + 3 + \alpha)(i + 4 + \alpha)} \cdot \frac{2(j+1-3k-2 \alpha)}{(j + 3 + \alpha)(j + 4 + \alpha)}   \end{align*}  \begin{align*}  + \frac{(\alpha + i + 1) (\alpha + i + 2)}{(i + 3 + \alpha)(i + 4 + \alpha)} \cdot \frac{2(j-3i-5-2 \alpha)}{(j + 3 + \alpha)(j + 4 + \alpha)}  \end{align*} \begin{align*} = \frac{-2 \alpha (11 + (5 -\alpha)(i + j) + 2 i j - 5 \alpha + 2 \alpha^2)}{(i + 3 + \alpha)(i + 4 + \alpha)(j + 3 + \alpha)(j + 4 + \alpha)}.  \end{align*}

\noindent For $i \geq j+1$, the computations are similar to those above.  \end{enumerate} \end{proof}

Since $P$ serves as an interrupter for the posinormal operator $M$, we may apply Theorem $1.2$ with $Q=I$ and conclude that $M$ is hyponormal if $I-P$ is a positive operator, and that will occur if all of its finite sections $S_n$ are positive.  Using Lemma $3.1 (c)$, it is straightforward to show that the first finite section of $I-P$ has determinant 
\begin{equation*} \det(S_0) = \frac{120+ 140 \alpha + 28 \alpha^2 + 8 \alpha^3 }{(3 + \alpha)^2 (4 + \alpha)^2},  \end{equation*}   

 \noindent the second finite section has determinant
 
  \begin{equation*} \det(S_1) = \frac{16 (2100 + 4340 \alpha + 2901 \alpha^2 + 1004 \alpha^3 +234 \alpha^4 + 38 \alpha^5 + 3 \alpha^6)}{(3 + \alpha)^2 (4 + \alpha)^4 (5 + \alpha)^2} , 
 \end{equation*} 
 
 \noindent the third finite section has determinant   
\begin{align*}   \det(S_2) = \frac{16 (1134000 + 3175200 \alpha + 3319710 \alpha^2 + 1884493 \alpha^3 + 686703 \alpha^4) }{(3 + \alpha)^2 (4 + \alpha)^4 (5 + \alpha)^4 (6 + \alpha)^2}  \end{align*}

\begin{align*}   + \frac{16 (178049 \alpha^5 + 34359 \alpha^6 + 4742 \alpha^7 + 408 \alpha^8 + 16 \alpha^9)}{(3 + \alpha)^2 (4 + \alpha)^4 (5 + \alpha)^4 (6 + \alpha)^2},  \end{align*}

\noindent and the fourth finite section has determinant 
\begin{align*}   \det(S_3) = \frac{16 (1047816000 + 3582532800 \alpha + 4875510240  \alpha^2 + 3747078072  \alpha^3) }{(3 + \alpha)^2 (4 + \alpha)^4 (5 + \alpha)^4 (6 + \alpha)^4 (7 + \alpha)^2}  \end{align*}

\begin{align*}   + \frac{16 (1885128129  \alpha^4 + 675769080  \alpha^5 + 182338742  \alpha^6 + 38146384  \alpha^7  )}{(3 + \alpha)^2 (4 + \alpha)^4 (5 + \alpha)^4 (6 + \alpha)^4 (7 + \alpha)^2} \end{align*}

\begin{align*}   + \frac{16 ( 6184561  \alpha^8 + 750976  \alpha^9 + 63688  \alpha^{10} + 3328  \alpha^{11} + 
     80  \alpha^{12})}{(3 + \alpha)^2 (4 + \alpha)^4 (5 + \alpha)^4 (6 + \alpha)^4 (7 + \alpha)^2}. \end{align*}

\noindent These calculations were assisted by [\textbf{11}].  Clearly each of these determinants is positive for all $\alpha \geq 0$, so we close with the following conjecture.

\begin{conjecture}  The operator $M \in \mathcal{B} (\ell^2)$ is also hyponormal for $\alpha \in (0,1)$.
\end{conjecture}

\noindent This conjecture would extend Corollary $2.4$ and thereby assert the hyponormality of the generalized Ces\`{a}ro matrices of order two for all $\alpha \geq 0$.

\end{document}